\documentclass[12pt]{amsart}

\usepackage{setspace}
\usepackage{amsmath, amssymb,amsthm, amsfonts, accents}
\newtheorem{thm}{Theorem}
\newtheorem{lemma}{Lemma}
\newtheorem{prop}{Proposition}
\newtheorem{cor}{Corollary}

\theoremstyle{definition}
\newtheorem{definition}{Definition}

\theoremstyle{remark}
\newtheorem{remark}{Remark}

\numberwithin{equation}{section}

\newcommand{\Ric}{\mbox{Ric}}
\newcommand{\R}{\mathbb R}

\newcommand{\be}{\begin{equation}}
\newcommand{\ee}{\end{equation}}
\newcommand{\bee}{\begin{equation*}}
\newcommand{\eee}{\end{equation*}}

\def\p{\partial}

\def\lf{\left}
\def\ri{\right}
\def\Pi{\displaystyle{\mathbb{II}}}

\def\m{\mathfrak{m}}

\def\tH{\tilde{H}}
\def\l{\lambda}
\def\e{\epsilon}

\def\tg{\tilde{g}}
\def\iM{\accentset{\circ}{M}}

\begin{document}

\title{Extensions and fill-ins with nonnegative scalar curvature}

\author{Jeffrey L. Jauregui}
\address[Jeffrey L. Jauregui]{Department of Mathematics, University of Pennsylvania, Philadelphia, PA 19104, USA.}
\email{jauregui@math.upenn.edu}

\author{Pengzi Miao}
\address[Pengzi Miao]{
Department of Mathematics, University of Miami, Coral Gables, FL 33146, USA.}
\email{pengzim@math.miami.edu}

\author{Luen-Fai Tam$^1$}
\address[Luen-Fai Tam]{The Institute of Mathematical Sciences and Department of
 Mathematics, The Chinese University of Hong Kong,
Shatin, Hong Kong, China.}
\email{lftam@math.cuhk.edu.hk}

\thanks{$^1$Research partially supported by Hong Kong RGC General Research Fund  \#CUHK 403011}

\renewcommand{\subjclassname}{
  \textup{2010} Mathematics Subject Classification}

\date{}

\begin{abstract}
Motivated by the quasi-local mass problem in general
relativity, we apply the asymptotically flat extensions, constructed by
Shi and Tam in the proof of the positivity of the Brown--York mass, 
to study a fill-in problem of realizing geometric data on a 2-sphere as the boundary
of a compact 3-manifold of nonnegative scalar curvature.  We characterize
the relationship between two borderline cases: one in which the Shi--Tam
extension has zero total mass, and another in which fill-ins of
nonnegative scalar curvature fail to exist. Additionally, we prove a type
of positive mass theorem in the latter case.
\end{abstract}

\maketitle

\section{Introduction}

The boundary geometry of compact Riemannian $3$-manifolds of  nonnegative scalar curvature
plays an important role in the quasi-local mass problem in general relativity.

Given a spacetime $N$ satisfying the dominant energy condition, and a bounded, space-like, time-symmetric (i.e., totally geodesic) hypersurface  $\Omega$ in $N$, the Riemannian manifold $\Omega$  necessarily has nonnegative scalar curvature.
The quasi-local mass of $\Omega$ in $N$ is
expected to depend only on the geometry of the boundary 2-surface $\Sigma = \partial \Omega$: in particular,
the induced metric $\gamma$  on $ \Sigma$ and the mean curvature $H$ of $ \Sigma $ in $\Omega$, in the outward direction.
When apparent horizons of black holes are present in $\Omega$, one assumes $ \p \Omega \setminus  \Sigma $
is not empty and consists of minimal surfaces.

In \cite{Bartnik89, Bartnik90}, Bartnik conjectured that,  given $\Omega$, $\Sigma$, $ \gamma $ and $ H$ in the above setting,
there exists  an asymptotically flat, static vacuum extension
of $\Omega$ with the prescribed  boundary metric $\gamma$ and boundary mean curvature $H$.
Motivated by that conjecture,  the following definition was given   in   \cite{Jauregui11}
(all metrics and functions below are assumed to be smooth, unless otherwise stated):

\begin{definition}
A triple  $(\Sigma, \gamma, H)$, where $ \Sigma$ is a surface
that is topologically   a $2$-sphere, $\gamma$ is a metric
on $ \Sigma$ of positive Gaussian curvature, and $H$ is a positive function on $ \Sigma$,
is called \emph{Bartnik data}.
\end{definition}

For Bartnik data $(\Sigma, \gamma, H)$ to be physically meaningful, one wants to know
whether it indeed arises  as the  boundary data of some compact $3$-manifold of
nonnegative scalar curvature. We recall the following definition from \cite{Jauregui11}:

\begin{definition}
A  \emph{nonnegative scalar curvature fill-in} of Bartnik data
 $(\Sigma, \gamma, H)$ is
 a compact Riemannian $3$-manifold with boundary $(\Omega, g)$ of nonnegative scalar curvature such that:
\begin{itemize}
\item $(\Sigma, \gamma)$ is isometric to a connected  component of $\p \Omega$
and, under this isometry, $H$ equals the mean curvature of $ \Sigma$ in $ (\Omega, g)$, and
\item $ \p \Omega \setminus \Sigma$ is either empty or else is a (possibly disconnected) minimal surface.
\end{itemize}
\end{definition}

A necessary  condition for $(\Sigma, \gamma, H)$ to admit a nonnegative scalar curvature fill-in
 is provided by \cite[Theorem 1]{ShiTam02} (also known as the positivity of Brown--York mass \cite{BY1, BY2}).

\begin{thm} [Shi--Tam \cite{ShiTam02}] \label{thm-ShiTam}
Let $ \Omega$ be a compact Riemannian $3$-manifold with nonnegative scalar curvature with boundary $\Sigma$.
Suppose the induced metric $\gamma$ on $ \Sigma$ has positive Gaussian curvature and the mean curvature $H$
of $ \Sigma$ in $\Omega$ is positive,  then
\be \label{eq-BY}
 \int_\Sigma H  d v_\gamma    \le \int_\Sigma H_0 d v_\gamma
 \ee
where $H_0$ is the mean curvature of the embedding of $(\Sigma, \gamma)$ in $ \R^3$ and $ d v_\gamma$ is the area
form on $(\Sigma, \gamma)$. Moreover, equality in \eqref{eq-BY} holds if and only if $\Omega$ is isometric
to a domain in $ \R^3$.
\end{thm}

\begin{remark} \label{rmk-ShiTam}
Though Theorem \ref{thm-ShiTam} is stated for a manifold $\Omega$ with $ \p \Omega = \Sigma$,
its conclusion still holds  for those $\Omega$ with  $ \p \Omega \setminus \Sigma  $ consisting of  minimal surfaces.
In that case one  replaces the use of the positive mass theorem \cite{SchoenYau79, Witten81}
 in the proof in \cite{ShiTam02} by the  Riemannian Penrose inequality \cite{Bray01, HI01}.
 Details of such an argument can be found in \cite[section 3.2]{Miao08}.
\end{remark}

In \cite{Jauregui11}, a unique  constant $ \lambda_0 $ is defined in association with any Bartnik data $(\Sigma, \gamma, H)$:
\be
\l_0 = \sup \Lambda
\ee
where

\vspace{.1cm}

\noindent $ \Lambda = \left\{ \lambda > 0 \ | \  (\Sigma, \gamma, \lambda H ) \right. $  {admits  a   nonnegative scalar  curvature} fill-in$\}$.

\vspace{.1cm}

\noindent It follows  immediately  from Theorem \ref{thm-ShiTam} and Remark \ref{thm-ShiTam} that
\be \label{eq-upper-bd-l}
\l_0   \le \frac{ \int_\Sigma H_0 d v_\gamma } {\int_\Sigma H d v_\gamma  } .
\ee
On the other hand,  it was proved  in \cite[Theorem 10]{Jauregui11} that $\lambda_0 > 0$ and
\be
 \Lambda = (0, \lambda_0 )  \ \mathrm{or} \
 (0, \lambda_0 ] ,
 \ee
and  conjectured  that  $ \l_0 \in \Lambda$ (\cite[Problem 2]{Jauregui11}).

\vspace{.3cm}

In this paper, we first give a more refined estimate of $\l_0$ than \eqref{eq-upper-bd-l}
(see Theorem \ref{thm-l-mu-0} in Section \ref{sec-apps}).
As a corollary, we  characterize Bartnik data $(\Sigma, \gamma, H)$ whose $\lambda_0 $ gives the equality in \eqref{eq-upper-bd-l}.

\begin{thm}  \label{thm-main}
Given Bartnik boundary data $(\Sigma, \gamma, H)$,
the constant $\l_0$ satisfies
\be  \label{eq-sup-l}
\lambda_0     = \frac{ \int_\Sigma H_0 dv_\gamma} {\int_\Sigma Hdv_\gamma }
 \ee
 if and only if $ H$ is a constant multiple of $H_0$.
 Here,  $H_0$ is the mean curvature of the isometric embedding of $(\Sigma, \gamma)$ in $ \R^3$.
\end{thm}

It follows immediately from Theorem \ref{thm-main} and (\ref{eq-upper-bd-l}) that
  condition  \eqref{eq-BY}  is {not}   sufficient for $ (\Sigma, \gamma, H)$ to possess a nonnegative scalar curvature fill-in:
given Bartnik data $ (\Sigma, \gamma, H)$ whose  $ H $ is not a constant multiple of $H_0$,
if we consider the related triple $(\Sigma, \gamma, \l H)$, where
  $  \l \in \lf( \lambda_0, \frac{ \int_\Sigma H_0} {\int_\Sigma H }  \ri] $,
then  $(\Sigma, \gamma, \l H)$ satisfies \eqref{eq-BY} but does not admit any nonnegative scalar curvature fill-in.

A  statement stronger than Theorem \ref{thm-main}, which also follows from  Theorem \ref{thm-l-mu-0}, is:

\begin{thm} \label{thm-no-fill-in}
Given  Bartnik data  $(\Sigma, \gamma, H)$ satisfying
\be \label{eq-no-fill-in}
 \int_\Sigma (H_0 - H) d v_\gamma = 0  \ \mathrm{but} \ H \neq H_0 ,
 \ee
 there exists a constant $ \epsilon >0$, depending on $ (\Sigma, \gamma, H) $,
  such that for any function $\tilde{H}$ on $ \Sigma $ satisfying
\be \label{eq-condition-H}
\tilde{H} > H - \epsilon ,
 \ee
$(\Sigma, \gamma, \tilde{H})$ admits no fill-ins of nonnegative scalar curvature.
\end{thm}

Our proof of Theorem \ref{thm-l-mu-0} (stated in Section \ref{sec-apps})  follows the  idea and  arguments in   \cite{ShiTam02}.
Basically, we  use  extensions of $(\Sigma,  \gamma, H)$ to study  its fill-ins.

\begin{definition} A complete non-compact Riemannian manifold $(N, g)$
 with nonempty boundary is called  an {\em extension} of $(\Sigma, \gamma, H)$ if $ \p N = \Sigma$,
  the induced metric from $g $ on $ \Sigma $ is $\gamma$ and the mean curvature of $ \Sigma $ in $(N, g)$, in the direction pointing into $N$, is $H$.
  \end{definition}
Given  Bartnik data $(\Sigma, \gamma, H)$,
a natural choice of asymptotically flat  extension $(N, g)$, having zero scalar curvature, of $(\Sigma, \gamma, H)$
was constructed in \cite{ShiTam02}.  In section \ref{sec-mass},  we consider the map
$$\Phi(H) = \m ( g) $$
where $ \m (g)$ is the total (ADM) mass (see \cite{ADM61}) of such an $(N, g)$ (with $\Sigma$ and $\gamma$ fixed) and show that $ \Phi  (\cdot) $ is continuous in  the $C^0$-topology  and is strictly monotone decreasing along  each path $\{ \l H \}_{\l >0 }$.
In section \ref{sec-apps}, we apply results on $ \Phi(\cdot)$ to the fill-in problem and prove Theorem \ref{thm-l-mu-0}.

In the remaining  part of the paper, section \ref{sec-static}, we prove a positive mass theorem for asymptotically flat nonnegative scalar
curvature extensions of $(\Sigma, \gamma, \l_0 H)$, and show that the zero mass case occurs  only if the extension is static vacuum (see Definition \ref{def_static}). 
Such a result is in   relation to the conjecture (\cite[Problem 2]{Jauregui11}) that $ (\Sigma, \gamma, \l_0 H)$ admits a nonnegative
scalar curvature fill-in.

\begin{thm} \label{thm-static-extension-l-0-intro}
Given Bartnik data $(\Sigma, \gamma, H)$, let $(M, g)$ be an  asymptotically flat extension of  $(\Sigma, \gamma, \l_0 H)$
with  nonnegative scalar curvature.
Then either $ (M,g)$ has  positive mass or else $(M, g) $ is static vacuum with zero mass.
\end{thm}

\section{Analysis of the mass functional} \label{sec-mass}

In this section,  $ \Sigma_0 $ is assumed to be  a compact, connected, strictly convex  hypersurface in $ \R^n $ with $ n \ge 3$. Let $\Omega_0$ be the domain bounded by $\Sigma_0$ and  let $\Sigma_r=\{x\in \R^n\setminus \Omega_0\ |\ d(x,\Sigma_0)=r\}$ where $ d(\cdot, \Sigma_0)$ is the Euclidean
distance to $ \Sigma_0$.
Then $\Sigma_r$, ${r\ge0}$,  are strictly convex, and $N = \R^n\setminus \Omega_0$ is foliated by $\{\Sigma_r\}$. We shall henceforth identify $N$
with $\Sigma_0 \times [0, \infty)$, with the Euclidean metric  $g_{_E}$ given by $dr^2+g_r$, where $g_r$ is the induced metric on $\Sigma_r$.

Consider the parabolic initial value problem:
 \be \label{eq-pde}
\left\{\begin{array}{rll}
\displaystyle H_0 \frac{\p u  }{\p r}   & =& u^2 \Delta_r u + \frac12 ( u - u^3) R_r,  \hbox{  on $ \Sigma_0 \times [0, \infty)$},
\\
u(x,0)& =&u_0(x),
\end{array}\right.
 \ee
 where $H_0$ is the mean curvature of $\Sigma_r$, $R_r$ is the scalar curvature of $\Sigma_r$, and $ u_0 $ is a positive function on $ \Sigma_0$.
The geometric meaning of \eqref{eq-pde} is that the metric $ g_u = u^2 dr^2 + g_r$ has zero scalar curvature and
$(N , g_u )$ is an extension of $ (\Sigma_0, g_0, H_0/u_0)$.

Given any  $ u_0 > 0 $ on $ \Sigma_0$, by \cite[Theorem 2.1]{ShiTam02} there exists  a unique positive solution $u$
 to \eqref{eq-pde}. Define
\be \label{eq-mur}
\m (u_0 ; r) =  \int_{\Sigma_r} H_0 ( 1 - u^{-1} ) d \sigma_r
\ee
and
\be \label{eq-mu}
\m (u_0) = \lim_{ r  \rightarrow \infty} \m ( u_0 ; r),
\ee
where $d\sigma_r$ is the Euclidean area form induced on $\Sigma_r$.
The following facts on $\m (u_0; r) $ and $ \m (u_0)$ were proved in
 Lemma 4.2 and Theorem 2.1 (c) in \cite{ShiTam02}:
\begin{enumerate}
\item[a)] $\m (u_0; r) $ is monotone decreasing in $r$ with
\be \label{eq-ShiTam-monotone}
\frac{d}{dr}  \m (u_0; r)  = - \frac12 \int_{\Sigma_r} R_r u^{-1} ( 1 - u)^2 d \sigma_r .
\ee
\item[b)]  The metric $ g_u = u^2 dr^2 + g_r $ is  asymptotically flat, with total mass equal to
 $\m (u_0) $, up to multiplication by a positive constant depending only on $n$.
\end{enumerate}

In what follows, we apply the monotonicity of $ \m(\cdot, r)$ to obtain  some additional properties on $\mathfrak{m}(\cdot)$.  We first
give a comparison principle for solutions to \eqref{eq-pde}:

\begin{lemma} \label{lem-comparison} $\;$
 \begin{enumerate}
   \item [(i)] Given two  positive functions  $u_0$ and $ v_0$ on $ \Sigma_0$,
let $ u $, $ v$ be  the solutions to \eqref{eq-pde} on $ \Sigma_0 \times [0, \infty)$ with initial conditions $u_0$, $ v_0$ respectively.
If $ u_0 \ge v_0 $, then $ u \ge v $  on $\Sigma_0 \times [0, \infty) . $
If in addition $ u_0 \neq v_0$, then $ u > v $ on $ \Sigma_0 \times (0, \infty)$.
   \item [(ii)] Let  $ \{ u_0^{(i) } \} $ and $ v_0$  be positive functions on $ \Sigma_0$,
and let $ \{ u^{(i)}  \} $, $ v$ be  the solutions to \eqref{eq-pde} on $ \Sigma_0 \times [0, \infty)$ with initial conditions $ \{ u_0^{(i)} \}$, $ v_0$ respectively. Suppose $\lim_{i\to\infty}\sup_{\Sigma_0}|u_0^{(i)}- v_0|=0$. Then for all $T>0$,
$$
\lim_{i\to\infty}\sup_{\Sigma_0\times[0,T]}|u^{(i)}- v|=0.
$$
 \end{enumerate}

\end{lemma}

\begin{proof} (i)
Let $ f = u - v $.  By \eqref{eq-pde} and subtracting, one has
\be \label{eq-pde-f}
\displaystyle H_0 \frac{\p f  }{\p r}   =  u^2 \Delta_r f  + w  f,
\ee
where
\bee
w =  [( u + v ) \Delta_r v ]  - \frac12 [  ( u^2 + uv + v^2 )  - 1 ] R^r .
\eee
For each fixed  $ T > 0 $, let  $C$ be a constant satisfying
\be \label{eq-w-H0-C}
C  \ge  \frac{|| w ||_{ L^\infty ( \Sigma_0 \times [ 0, T ] )}}{   \min_{ \Sigma_0 \times [0, T] } H_0} .
\ee
Clearly $C$ can be chosen to depend only on
$$\sup_{\Sigma_0 \times [ 0, T ]}\lf(|\Delta_rv|+|u|+|v|+R^r+H_0^{-1}\ri). $$
Let $ \beta  $ be a constant satisfying $ \beta > C $.
On $\Sigma_0 \times [0,T]$, \eqref{eq-pde-f}  implies
\be \label{eq-mpde}
H_0 \frac{\p}{\p r} (e^{- \beta r  }f)  = u^2 \Delta_r(e^{- \beta r }f)  - ( \beta H_0 - w)   e^{- \beta  r }f,
\ee
and  $ \beta  H_0  - w >  0 $ on $\Sigma_0 \times [0, T]$ by \eqref{eq-w-H0-C}.
It follows from \eqref{eq-mpde} and the maximum principle that
 $ e^{-\beta r}f \ge 0 $ on $ \Sigma_0 \times [0, T]$. As $ T $ is arbitrary, $ f \ge 0 $ on $ \Sigma_0 \times [0, \infty)$.
If  $ f \neq 0$ on $ \Sigma_0$, then  $ f > 0 $ on $ \Sigma_0 \times (0, \infty)$  by  the strong maximum principle.

(ii) The proof is similar.
Since $ \{ u_0^{(i)} \} $ converges to $ v_0$ on $ \Sigma_0$ in the $C^0$-norm,
by \cite[Lemma 2.2]{ShiTam02} there is a constant $C_1$ independent of $i$ such that $|u^{(i)}|\le C_1$ on $\Sigma_0\times[0,\infty)$ for all $i$.
 Let $f^{(i)}=u^{(i)}-v$. For each fixed $T>0$,
 as in (i) one can find a constant $\beta>0$ which is independent of $i$ such that
\be \label{eq-mpde-1}
H_0 \frac{\p}{\p r} (e^{- \beta r  }f^{(i)})  = { (u^{(i)} ) }^2 \Delta_r(e^{- \beta r }f^{(i)})  - ( \beta H_0 - w^{(i)})   e^{- \beta  r }f^{(i)},
\ee
and $\beta H_0-w^{(i)}>0$ on $\Sigma_0 \times [0, T]$, where
\bee
w^{(i)}=  [( u^{(i)} + v ) \Delta_r v ]  - \frac12 [  ( u^{(i)})^2 +   u^{(i)}v + v^2   - 1 ] R^r .
\eee
One then concludes by the maximum principle that
$$
\max_{\Sigma_0\times[0,T]}\lf|e^{- \beta r  }f^{(i)}\ri|=\max_{\Sigma_0}|f^{(i)}|.
$$
From this the result follows.
\end{proof}

\begin{thm}\label{thm-mass-comparison}
Suppose $u_0 \ge v_0 $ on $ \Sigma_0$. Then
\begin{enumerate}
\item[(i)] $\mathfrak{m}(u_0;r) \ge \mathfrak{m}(v_0;r)$  for all $r\geq 0$. If $ \m (u_0; r) = \m (v_0; r) $ for some $r$, then $ u_0 = v_0$.

\item[(ii)] $\mathfrak{m}(u_0) \ge \mathfrak{m}(v_0)$. If  $\mathfrak{m}(u_0) = \mathfrak{m}(v_0)$, then   $u_0 = v_0$.

\end{enumerate}

\end{thm}
\begin{proof} (i) follows directly from Lemma \ref{lem-comparison}.
The  claim $ \m(u_0 ) \ge \m(v_0)$ in (ii)  follows from (i). We now focus on the equality case in (ii).

Let $ u$, $v$  be the unique positive solutions to \eqref{eq-pde} with initial conditions $u_0$,  $v_0$ respectively.
By Theorem 2.1  in \cite{ShiTam02}, $u$ and $v$ have an  asymptotic expansion
\be \label{eq-expan}
u = 1 + c(n) { \m (u_0) }{ r^{2-n} } + O ( r^{1-n} )
\ee
and
\be
v = 1 + c(n) { \m (v_0) }{ r^{2-n} } + O ( r^{1-n} )
\ee
 as $r \rightarrow \infty$, where $ c(n) $ is some positive constant depending only on $n$.
By Lemma 2.1 in \cite{ShiTam02}, the mean curvature $H_0$ of $\Sigma_r $ in $ \R^n$ satisfies
\be \label{eq-H-est}
 H_0 = \frac{n-1}{r} + O(r^{-2} ) .
 \ee
Therefore, it follows from \eqref{eq-expan} -- \eqref{eq-H-est} that
there exist constants $ R_1 > 0 $ and $ C_1 > 0 $ such that
\be \label{eq-mr-1}
\begin{split}
 \m (u_0; r )  - \m (v_0; r)  = & \ \int_{\Sigma_r} H_0 \lf( \frac{u - v}{u  v } \ri)  d \sigma_r \\
\ge & \  \frac{C_1}{r} \int_{\Sigma_r} (u  - v ) d \sigma_r ,
\end{split}
\ee
 for all $r\ge  R_1$. (Here we also used the fact $ u \ge v $ by Lemma \ref{lem-comparison}.)

Now suppose $ \m (u_0) = \m(v_0) $.
By \cite[Lemma 4.2]{ShiTam02}, i.e. \eqref{eq-ShiTam-monotone},  we  have
\be
\m (u_0; r)- \m  (v_0; r )
=  \frac12 \int_r^\infty \lf[\int_{\Sigma_s} R_s   \lf( 1 - \frac{1}{u v} \ri) (u - v )   d\sigma_s \ri] ds
\ee
for any $ r $.
By Lemma 2.1 in \cite{ShiTam02}, the scalar  curvature $R_r$ of $\Sigma_r $ satisfies
\be
 R_r = \frac{(n-1)(n-2)}{r^2} + O(r^{-3} ) .
\ee
Therefore, there  exist constants $ R_2 > 0 $ and $ C_2 > 0 $  such that
\be \label{eq-mr-2}
\m (u_0; r)- \m  (v_0; r )
\le    C_2 \int_r^\infty \lf[ s^{-n} \int_{\Sigma_s} (u - v)   d\sigma_s \ri] ds
\ee
for all $r \ge R_2$.
It follows from  \eqref{eq-mr-1} and \eqref{eq-mr-2} that
\be
\m  (u_0; r) - \m (v_0;r ) \le C_3 \int_r^\infty s^{1-n}  [\m  (u_0; s) - \m (v_0;s ) ] d s
\ee
for  $r \ge R_3 = \max \{ R_1, R_2 \}$  with  a constant $C_3 > 0$ independent of $r$.
Let
$$ F(r) =  \int_r^\infty s^{1-n} [ \m  (u_0; s) - \m (v_0;s ) ] d s, $$
which is finite by (\ref{eq-mr-2}).
Then
\bee
 r^{n-1} \frac{d }{dr} F (r) + C_3 F(r)  \ge 0 ,
\eee
which shows  $ F( r )  \exp\left(    \frac{C_3}{(2-n)r^{n-2}}\right)    $ is monotone increasing on $(R_3, \infty)$.
On the other hand, we have
$ F(r) \ge 0 $ and  $
\lim_{r \rightarrow \infty} F (r) = 0  .
$
Therefore,
$$ F(r) = 0 , \text{ for all}  \ r\ge R_3  .$$
Hence, $ \m (u_0; r) = \m (v_0; r)  $ for large $r$. We conclude   $ u_0 = v_0 $ by (i).
\end{proof}

Next we show that $ \m (\cdot)$ is continuous on $C^{\infty}_+(\Sigma_0)$, the space of smooth, positive functions on $\Sigma_0$,
endowed with the $C^0$-topology.

\begin{thm}\label{thm-mass-continuous}
The functional
$$ \m (\cdot) : C^{\infty}_+(\Sigma_0)\to \R$$ is continuous with respect to the $C^0$-norm.
\end{thm}
\begin{proof} Let $v_0\in C^{\infty}_+(\Sigma_0)$. Let  $\delta_0>0$ and $a>0$ be two constants such that 
if   $u_0\in   C^{\infty}_+(\Sigma_0)$ with $ ||u_0-v_0||_{C^0( \Sigma_0) }<\delta_0$, then $a^{-1}\le u_0\le a$. 
Given any   $  \delta \in (0, \delta_0) $, 
let $u_0\in   C^{\infty}_+(\Sigma_0)$ with $ ||u_0-v_0||_{C^0( \Sigma_0) }<\delta$. 
Let $u$ and $v$ be the positive solutions to \eqref{eq-pde} with initial data $u_0$ and $v_0$ respectively. 
 By  \cite[Lemma 2.2]{ShiTam02}, there is a constant $C_1$, depending only on $\Sigma_0$ and $ a$
(hence  independent of $u_0$)
such that
\be \label{eq-uand1}
|u-1| \le {C_1} { r^{2-n}} \ \ \mathrm{and} \ \ |v-1|\le {C_1}{r^{2-n}}. 
\ee
 By  \cite[Lemma 2.1]{ShiTam02},  the scalar curvature $ R_r$ of $ g_r$ satisfies
\be \label{eq-R-est}
0<R_r\le  {C_2} {r^{-2}}
\ee
for some constant $ C_2$ depending only on $ \Sigma_0$.
By \cite[Lemma 2.4]{ShiTam02}, $(\Sigma_0, r^{-2} g_r)$ converges  to the standard unit sphere $ \mathbb{S}^{n-1}$. 
Therefore, by \cite[Lemma 4.2]{ShiTam02}, i.e. \eqref{eq-ShiTam-monotone},
 we have for any $r_0 \geq 0$
\be\label{eq-mass-continuous-2}
\begin{split}
& \ \lf|\lf[\mathfrak{m}(u_0)-\mathfrak{m}(v_0)\ri]-\lf[\mathfrak{m}(u_0;r_0)-
\mathfrak{m}(v_0 ;r_0)\ri]\ri|\\
= & \ \frac12 \lf|\int_{r_0}^\infty\lf[\int_{\Sigma_s}R_s\lf(1-\frac1{uv}\ri)(v-u)
  d\sigma_s \ri]ds \ri|\\
  \le & \ C_3 r_0^{2 - n} 
  \end{split}
\ee
for some constant $C_3$ independent of $u_0$ and $r_0$. 

Given any $ \e > 0 $,  choose $r_0$ large enough so that $C_3 r_0^{2 - n} <  \frac12\e$. 
For this  fixed $ r_0$,  Lemma \ref{lem-comparison}(ii) implies that 
 the functional  $u_0 \mapsto \m(u_0;r_0)$  is continuous with respect to the $C^0$-norm. Hence if $\delta>0$ is small enough, then 
$$\lf|\m(u_0;r_0)-\m(v_0;r_0)\ri| < \frac12\e .$$ 
Combining this with \eqref{eq-mass-continuous-2}, the theorem follows. 
\end{proof}

By Theorem \ref{thm-mass-comparison} and Theorem \ref{thm-mass-continuous}, we have the following:

\begin{thm} \label{thm-monotone-t}
Given any $  \ u_0 \in C^\infty_+ (\Sigma_0)$,  the function
$ t \mapsto \m (t u_0)  $
is a continuous, strictly increasing function on $(0, \infty)$,  satisfying
\be \label{eq-limit-tu0}
 \lim_{t \rightarrow 0+ }  \m (t u_0) < 0 \ \mathrm{and} \
\lim_{t \rightarrow \infty} \m (t u_0) > 0 .
\ee
Moreover, if $t_0$ is the unique number in $(0,\infty)$ such that $\m (t_0 u_0)=0$, then
\be \label{eq-t-0}
t_0\ge \frac{\int_{\Sigma_0}H_0u_0^{-1} \ d \sigma_0 }{\int_{\Sigma_0}H_0  \ d \sigma_0 },
\ee
and equality holds if and only if $u_0$ is a constant.
\end{thm}
\begin{proof}
The continuity and strict monotonicity of $ \m ( t u_0)$ follow from Theorem \ref{thm-mass-continuous} and
Theorem \ref{thm-mass-comparison},  respectively. Claims \eqref{eq-limit-tu0} follow from  Theorem \ref{thm-mass-comparison}
and  the fact  $ \m (v_0) = 0 $ if $v_0\equiv1$.
Let $t_0$ be such that $\m (t_0 u_0)=0$. By \eqref{eq-ShiTam-monotone}, we have $\m (t_0 u_0;0)\ge0$ and
$\m (t_0 u_0;0) = 0$  if and only if $t_0u_0=1$.
Since
$\m (t_0 u_0;0)=\int_{\Sigma_0}H_0\lf[1-(t_0u_0)^{-1}\ri]d\sigma_0,$
the second part of the theorem follows.
\end{proof}

\section{Applications to fill-ins} \label{sec-apps}
Given  Bartnik data $(\Sigma, \gamma, H)$, we isometrically embed $(\Sigma, \gamma)$ as a closed, strictly convex surface
$\Sigma_0$ in $ \R^3$ with mean curvature $H_0$, and we identify its exterior region  with
$N = \Sigma \times [0, \infty)$.
Let $ \m (\cdot; r) $ and $\m (\cdot) $  be the  functionals defined in \eqref{eq-mur} and \eqref{eq-mu}.
Apply  Theorem \ref{thm-monotone-t} to the function $u_0=H_0/H$, so that the corresponding metric $u^2 dr^2 + g_r$
is an extension of $(\Sigma, \gamma, H)$.  In particular, there is a unique $ \mu_0>0$  such that
\be\label{eq-mu-zero}
\left\{
  \begin{array}{lll}
    \m( H_0 / ( \mu_0 H) ) & =& 0, \ \hbox{ } \\
    \m(H_0 / ( \mu  H)  ) & > & 0, \ \hbox{if $\mu<\mu_0$,} \\
    \m(H_0 / ( \mu  H) ) & < & 0,  \ \hbox{if $\mu>\mu_0$.}
  \end{array}
\right.
\ee

\begin{thm} \label{thm-l-mu-0}
Given Bartnik data $(\Sigma, \gamma, H)$, the following are true:
 \begin{enumerate}
   \item [(i)]
$ \displaystyle 
\l_0  \le \mu_0  \le  \frac{ \int_\Sigma H_0 d v_\gamma } {\int_\Sigma H d v_\gamma }.$ 

   \item [(ii)] $\mu_0 =  \frac{ \int_\Sigma H_0 d v_\gamma } {\int_\Sigma H d v_\gamma }$ if and only if $ H=\mu_0^{-1}H_0 $.

 \item[(iii)] For any function  $ \hat{H} \ge \mu_0 H $,
$(\Sigma, \gamma, \hat{H})$ does not have a fill-in of nonnegative scalar curvature, unless
$ \hat{H} = \mu_0 H = H_0 $.

\item [(iv)] If  $\l_0 = \mu_0$, then  $(\Sigma, \gamma , \lambda_0  H)$ does not admit  a nonnegative scalar curvature fill-in,
unless  $ \mu_0 H=H_0$.
 \end{enumerate}
\end{thm}
\begin{proof}  (i) For $\lambda>0$, if $(\Sigma, \gamma, \lambda H)$   has a nonnegative scalar curvature fill-in, then
$\m( H_0/ (\l H) ) \ge0$ by the version of the positive mass theorem in  \cite[Theorem 3.1]{ShiTam02} (also see \cite[Theorem 1]{Miao02} and Remark \ref{rmk-ShiTam}).
Therefore, $ \l \le \mu_0 $. This  implies the first inequality, by the definition of $\lambda_0$. The second inequality
follows from \eqref{eq-t-0}.

 (ii) is  just a restatement of the rigidity case in \eqref{eq-t-0}.

(iii) Given $ \hat{H} \ge \mu_0 H $, suppose $ (\Sigma, \gamma, \hat{H})$ has a nonnegative scalar curvature fill-in $(\Omega, g)$.
Using the fact $ \m (H_0 / (\mu_0 H) ) = 0 $ and applying the rigidity part of
\cite[Theorem 3.1]{ShiTam02} (also see \cite[Theorem 2]{Miao02} and Remark \ref{rmk-ShiTam}),
we conclude $ \hat{H} = \mu_0 H$ and $ (\Omega, g)$ is isometric to a domain in $ \R^3$ with boundary mean curvature $H_0$.
Therefore, $ \hat{H} = \mu_0 H = H_0 $.

(iv) Suppose  $\lambda_0 =\mu_0 $ and $(\Sigma, \gamma , \lambda_0 H)$ has a nonnegative scalar curvature fill-in,
then  the proof of (iii) shows $\mu_0 H = H_0 $.
 \end{proof}

\begin{remark}
In \cite{Jauregui11} and \cite{Miao08} the following quasi-local mass functionals were proposed:
\begin{align*}
m_0(\Sigma, \gamma, H) &= \sqrt{\frac{|\Sigma|_\gamma}{16\pi}}\left( 1- \frac{1}{\lambda_0^2}\right),\\
m_1(\Sigma, \gamma, H) &= \sqrt{\frac{|\Sigma|_\gamma}{16\pi}}\left[ 1-    \left( \frac{ \int_\Sigma H dv_\gamma }{ \int_\Sigma H_0 dv_\gamma } \right)^2 \right],
\end{align*}
where $|\Sigma|_\gamma$ is the area of $(\Sigma, \gamma)$.  In terms of $\mu_0$, a related functional  is:
\begin{equation}
m_2(\Sigma, \gamma, H) = \sqrt{\frac{|\Sigma|_\gamma}{16\pi}}\left( 1- \frac{1}{\mu_0^2}\right),
\end{equation}
and Theorem \ref{thm-l-mu-0} shows $m_0 \leq m_2 \leq m_1$.
\end{remark}

If $  k H=   H_0$ for some positive constant $k$, then  $(\Sigma, \gamma, k H)$ has a Euclidean fill-in $(\Omega_0, g_{_E})$, hence
$ \l_0 = \mu_0 = k$.
When $ H$ is not a constant multiple of $H_0$, Theorem \ref{thm-l-mu-0} directly implies
the following corollary, whose (i) and (ii) are Theorem \ref{thm-main} and \ref{thm-no-fill-in} in the introduction.

\begin{cor}\label{thm-no-fill-ins-1}
Given  Bartnik data  $(\Sigma, \gamma, {H})$ satisfying
\be \label{eq-no-fill-in-1}
 \int_\Sigma (H_0 - {H}) d v_\gamma = 0  \ \mathrm{but} \ {H} \neq H_0 ,
 \ee
then
\begin{enumerate}

  \item [(ii)] $\lambda_0 \le \mu_0 <1$.

  \item [(ii)] There exists  $ \epsilon >0$, depending on $ (\Sigma, \gamma, {H}) $,
  such that if
 $ \hat{H} > {H} - \epsilon $,
then $(\Sigma, \gamma, \hat{H})$ admits no fill-ins of nonnegative scalar curvature.

\item[(iii)] If $\lambda_0=\mu_0$, then $(\Sigma, \gamma, \l_0 H) $  does not admit any nonnegative scalar curvature fill-in.

\end{enumerate}

\end{cor}

We record  a  lower bound of  $\mu_0$.
 \begin{prop}\label{prop-mu-lower-bound}
 Given Bartnik data $(\Sigma, \gamma, H)$,
 \be  \label{eq-mu-bound}
 {\mu}_0^2 \ge  \lf( \int_\Sigma H_0 dv_\gamma \ri) \lf( \int_\Sigma \frac {H^2}{H_0} dv_\gamma\ri)^{-1}. \ee
\end{prop}
 \begin{proof}  By \cite[Lemma 2.11]{ShiTam02} (also see \cite{Bartnik93}),
if $u$ is the solution of \eqref{eq-pde}  with initial condition $ u_0$,
then $$   \int_{\Sigma_r} H_0 ( 1 - u^{-2} ) d \sigma_r $$ is monotone nondecreasing in $r$
 and limiting to a positive constant multiple of  $\m ( u_0)$.
 Take $ u_0 = H_0 / ( \mu_0 H) $, we have
 $$
 \int_{\Sigma} H_0 ( 1 - (\mu_0H)^2H_0^{-2} ) d \sigma \le 0,
 $$
by the definition of $\mu_0$.  This implies \eqref{eq-mu-bound}.
 \end{proof}

\section{$\l_0$ and static vacuum metrics} \label{sec-static}

Given  Bartnik data $(\Sigma, \gamma, H)$,
 it was conjectured  that $(\Sigma, \gamma, \l_0 H)$ admits a nonnegative scalar curvature fill-in (cf. \cite[Problem 2]{Jauregui11}).
In \cite[Proposition 5]{Jauregui11}, it was proved that such a fill-in must be static vacuum.
 By results in \cite{Miao02, ShiTam02} and Remark 1, a necessary condition for this conjecture to hold is that
  any asymptotically flat nonnegative scalar curvature extension of $(\Sigma, \gamma, \l_0 H)$ must have nonnegative mass.
In the rest of this paper we prove that this necessary condition holds.

\begin{definition}[\cite{Corvino}]
\label{def_static}
A Riemannian metric $g$ is called \emph{static} on an open set $\Omega $ if the formal $L^2$-adjoint of the
linearization of the scalar curvature map at $g$ has non-trivial kernel, i.e. there exists a function $f$,
not identically zero, satisfying
\be
 - ( \Delta_g f) g + \nabla^2_g f - f \Ric (g) = 0
\ee
on $ \Omega$. Here, $ \Delta_g  f $ and $ \nabla^2_g  f $ are the Laplacian and Hessian of $f$, and $\Ric(g)$ is the Ricci curvature of $g$.
$g$ is called \emph{static vacuum} if, in addition, $g$ has zero scalar curvature.
\end{definition}

\begin{thm} \label{thm-static-extension-l-0}
Given Bartnik data $(\Sigma, \gamma, H)$, let $(M, g)$ be an  asymptotically flat extension, having nonnegative scalar curvature,
 of  $(\Sigma, \gamma, \l_0 H)$.  Then the total mass of $(M,g)$ is nonnegative, and is zero only if $g$ is static vacuum on the interior of $M$.
\end{thm}

\begin{proof}  If the statement fails, then $\m(g) =0$ and $g$ is not static vacuum on $\iM$ (the interior of $M$),
or else $\m(g) < 0$.  We address these two cases separately.  $\m(\cdot)$ will denote the total mass.

First, suppose $\m(g) = 0$ and $ g $ is not static vacuum on $\iM$.  Then $g$ is not static on $\iM$, since static metrics have constant
scalar curvature (\cite[Proposition 2.3]{Corvino}) and $g$ is asymptotically flat.
Then there exists a domain
$U \subset \subset \iM $ such that $ g$ is not static on $U$
(see the proof of Proposition 3.2 in \cite{Corvino}).
Applying Theorem 4 in \cite{Corvino}, one can  push the scalar curvature up slightly on $U$ while keeping
the metric $g$ unchanged outside $\overline{U}$. In other words, there exists another metric
$\tg$ on $M$ such that
$\tg$ agrees with $g$ outside $\overline{U}$,
the scalar curvature of $\tg$, denoted by $R(\tg)$, is nonnegative everywhere,
and $ R(\tg) $ is strictly positive somewhere in $U$.  Observe $\m(\tg) = \m(g)$.

On $(M, \tg)$, let $ w $ be the unique positive solution to
$$
\left\{
\begin{array}{rl}
   L_{\tg} w  & =  0 \ \ \text{ in } M, \\
   w & =  1 \  \ \text { on } \Sigma,\\
   w  & \to  1 \  \ \text { at infinity,}
\end{array}
\right.
$$
where $ L _{\tg} = -8\Delta_{\tg} + R(\tg)$ is the conformal Laplacian.
 The conformally deformed metric $ \hat{g} = w^4 \tg $ has zero scalar curvature.
 The total masses $ \m (\hat{g})$ and $\m  (\tg)$ are related by
 $  \m  (\hat{g}) =  \m  (\tg)  + 2 A $,
 where $A$ is the constant in the asymptotic expansion
 $ w = 1 + \frac{A}{| x|} + O ( | x |^{-2} ) $, $ |x| \rightarrow \infty$, on $M$.
Since $ R(\tg) \ge 0 $  and $ R(\tg)$ is not identically zero, we have $ A< 0$ by the
strong maximum principle.  Hence,
\be \label{eq-mass-relation}
  \m  (\hat{g}) < \m  (\tg) = \m(g) =0 .
  \ee
  On the other hand, the mean curvature  $\hat{H}$, $\tilde{H} $ of $\Sigma $ in $(M, \hat{g})$, $(M, \tg)$
  are related by $ \hat{H} = \tH + 4 \frac{\p w}{\p \nu} $,
where $\nu$ is the $\tg$-unit normal to $ \Sigma$ pointing into $M$ and   $  \frac{\p w}{\p \nu}  < 0 $ by the strong maximum principle.
 Therefore,
\be \label{eq-mean-H}
\hat{H} < \tH .
\ee

Now take a   sequence $\{ \l_k \}$ such that $ \l_k < \l_0$,
$ \lim_{k \rightarrow \infty} \l_k = \l_0$
and note $ (\Sigma, \gamma, \l_k H )$ admits a nonnegative scalar curvature fill-in $(\Omega_k , g_k)$ for each $k$.
Because $ \tg $ and $ g $ agree outside $ \overline{U}$, we have $ \tH = \l_0 H $.
By \eqref{eq-mean-H}, we have
\be \label{eq-bdry-H}
\l_k H >  \hat{H}
\ee
if $k$ is sufficiently  large.
For such a $k$, we attach $(\Omega, g_k)$ to $(M,  \hat{g})$  along $ \Sigma $.
Applying  results  in \cite{Miao02, ShiTam02} (also see Remark \ref{rmk-ShiTam}), we  have
\be
\m  (\hat{g}) \ge 0 ,
\ee
which contradicts \eqref{eq-mass-relation}.  This completes the first part of the proof.

To handle the remaining case, suppose that $ \m(g) < 0$.
For a number $\epsilon \in (0,1)$, let $u_{\epsilon}$ be the unique, positive solution to:
$$
\left\{
\begin{array}{rl}
   \Delta_g u_{\epsilon} & =  0 \ \ \text{ in } M,\\
   u_{\epsilon} & =  1 \  \ \text { on } \Sigma,\\
   u_{\epsilon} & \to  1-\epsilon  \  \ \text { at infinity}.
\end{array}
\right.
$$
The conformal  metric ${g}_\e = u_{\epsilon}^4 g$ has nonnegative scalar curvature.
The total masses $ \m ({g_\e})$ and $\m  (g)$ satisfy
 $  \m  ({g_\e}) =  \m  (g)  + 2 (1 - \epsilon) A_\epsilon$,
 where $A_\epsilon $ is  in the  expansion
 $ u_\epsilon = (1- \epsilon)  + \frac{A_\epsilon}{| x|} + O ( | x |^{-2} ) $, $ |x| \rightarrow \infty$.
 Since $ \Delta_g u  = 0 $,
 $  A_\epsilon =  -\frac{1}{4 \pi} \int_\Sigma \frac{\p u_\epsilon}{\p \nu} $
 where $\nu$ is the $g$-unit normal to $ \Sigma$ pointing into $M$.
 Hence, $ A_\epsilon > 0 $ by the strong maximum principle, and  $ \lim_{\epsilon \rightarrow 1 }  A_\epsilon = 0 $.
Therefore,
\be
\m ({g_\e} ) < 0
\ee
 for all $\epsilon$ sufficiently small.  Fixing such a value of $\epsilon$,
 $(M, g_\e) $ is a negative-mass extension of $(\Sigma, \gamma, \hat H)$,
 where $\hat H = H + 4 \frac{\p u_\epsilon}{\p \nu} < H$.  Repeating a similar argument as in
the first part of the proof yields a contradiction.
\end{proof}

\begin{remark}
The proof of Theorem  \ref{thm-static-extension-l-0}  establishes  the following:
Suppose  $\{ (\Sigma, \gamma, H_k) \}$ is a sequence of Bartnik data such that
$(\Sigma, \gamma, H_k)$ admits a nonnegative scalar curvature fill-in for each $k$ and $ \{ H_k \}$ converges to a smooth, positive
function $\tilde{H}$ in the $C^0$-norm. Then given any  positive function  $H$ satisfying $ H \le \tilde{H}$,
 the total mass of  an  asymptotically flat, nonnegative scalar curvature  extension of $(\Sigma, \gamma, H)$
is  nonnegative, and is zero only if the extension is static vacuum.
\end{remark}

\begin{cor}
Given Bartnik data $(\Sigma, \gamma, H)$, suppose $ \l_0 = \mu_0$.
Then
the asymptotically flat metric
$ g = u^2 d r^2 + g_r $,
where $ u $ is the unique positive  solution to \eqref{eq-pde}
with initial condition
$  H_0 / ( \mu_0 H)  $,
is static vacuum on $  \Sigma_0 \times (0, \infty)$.
\end{cor}

\begin{proof}
This  follows directly from  the fact $\m ( g ) = 0 $ and Theorem \ref{thm-static-extension-l-0}.
\end{proof}

\end{document}